\documentclass[times]{my_iapress}

\def\volumeyear{201x}

\usepackage{amsmath,amssymb}
\usepackage{graphicx}

\usepackage{url}
\usepackage{amssymb}
\usepackage{amsmath}
\usepackage{subfig}

\newtheorem{remark}{Remark}

\newtheorem{theorem}{Theorem}

\begin{document}

\runningheads{\ A simple mathematical model for unemployment: 
a case study in Portugal with optimal control}{A. Galindro and D. F. M. Torres}

\title{A simple mathematical model for unemployment:\\ 
a case study in Portugal with optimal control\footnote[2]{Part 
of first author's M.Sc. thesis, carried out at 
University of Aveiro under the Master Programme 
\emph{Mathematics and Applications}.}}

\author{An\'ibal Galindro\affil{1},
Delfim F. M. Torres\affil{2}\corrauth}

\address{\affilnum{1}Centre for Transdisciplinary Development Studies,
University of Tr\'as-os-Montes and Alto Douro, Polo II -- ECHS, 
Quinta de Prados, 5000-801 Vila Real, Portugal.
\affilnum{2}Center for Research and Development in Mathematics and Applications (CIDMA),
Department of Mathematics, University of Aveiro, 3810-193 Aveiro, Portugal.}

\corraddr{delfim@ua.pt (Delfim F. M. Torres).}


\begin{abstract}
We propose a simple mathematical model for unemployment.
Despite its simpleness, we claim that the model is more 
realistic and useful than recent models available in the literature. 
A case study with real data from Portugal supports our claim.
An optimal control problem is formulated and solved,
which provides some non-trivial and interesting conclusions.
\end{abstract}



\maketitle


{\small \noindent{\bf Mathematics Subject Classification 2010\ \ } 
91B64, 91B74 (Primary); 34C60, 49K15 (Secondary).}


\section{Introduction}

Unemployment is an extremely severe social and economic problem, 
born from the difference between demand and supply of the labour market 
and sometimes emphasized by population's growth. The unemployed population can be defined 
as the portion of able citizens that desire to work (known as the active population) 
but, unfortunately, due to insufficient supply, are deprived from working. 
Generally, unemployment is a precarious social situation since a portion 
of the population normally struggles to maintain a minimum welfare 
and consumption level. Simultaneously, regarding the macroeconomic perspective, 
higher unemployment rates intensify the pressure on social protection measures, 
e.g., unemployment subsidies, and, consequently, the associated government expenditure. 
There are numerous policies that interact directly with a country's level of unemployment 
or, symmetrically, with the level of employment offered to the citizens. 
As an example, we have the establishment of a minimum wage and currency devaluation. 
Here, we propose and investigate a simple new model that describes 
well the current labour market in Portugal. 

Regarding available literature focused on the subject, we go back to 2011, 
time when Misra and Singh \cite{ms1} proposed a non-linear mathematical 
model of unemployment based on Nikolopoulos and Tzanetis previous work of 2003 \cite{nt}.
More precisely, they suggested a model for housing allocation of a homeless population 
due to a natural disaster, described by a system of ordinary differential equations 
with the following three variables: unemployed population, employed individuals, 
and temporarily workers \cite{ms1}. They analyse the equilibrium of the model,  
using the stability theory for differential equations, and perform a few numerical simulations, 
concluding that the unemployment battle may need immediate measures: they predict 
that the unemployment rate may rise quickly and, if those high unemployment 
values are reached, then it might be very difficult to overcome that much bigger 
problem in the future \cite{ms1}. Misra and Singh developed further their empirical work, 
and in 2013 they replaced the temporarily employed variable by newly created vacancies 
with a delayed feature \cite{ms2}. Another non-linear mathematical model of unemployment 
was then proposed by S\^{\i}rghi et al. in 2014, based on a differential 
system with distributed time delays \cite{sa}. Moreover, they also considered 
the unemployment level as a signal to the employers to hire at a lower wage \cite{sa}. 
The split of the vacancies variable into current and government created vacancies, 
was also a significant modification to the previous Misra and Singh models \cite{ms1,ms2}. 
Recently, Harding and Neam\c{t}u also followed the ideas of Misra and Singh \cite{ms1,ms2}, 
extending further the previous efforts by presenting an unemployment model where job search 
is open both to native and migrant workers \cite{hn}. They consider two policy approaches: 
the first that aims to reduce unemployment by observing both past values of unemployment and migration, 
the second considering the past values of unemployment alone \cite{hn}. Similarly, by considering 
the previously mentioned bibliography, Munoli and Gani (2016) seek to define 
an optimal control policy to unemployment through two possible measures: 
government policies, focused at providing jobs directly to unemployed persons,  
and government policies, aiming the creation of new vacancies \cite{mg}. 
The dynamics \cite{mg} of the (un)employment market are defined by three differential equations, 
referring those to variation on unemployed/employed people and vacancies available, 
quite similarly to Misra and Singh (2013) \cite{ms2}. Here, taking into account 
the Portuguese unemployment reality between 2004 and 2016,
we propose a few changes in Munoli and Gani (2016) model \cite{mg}
and apply optimal control to discuss suitable policies for avoiding unemployment.

The remainder of the paper is organized as follows. Section~\ref{sec:2}
presents monthly real data for the unemployed Portuguese population 
in the time period from January 2004 to June 2016. 
We show that the model of \cite{mg}, suggested by Munoli and Gani in 2016,
is not suitable to describe the Portuguese reality. 
In Section~\ref{sec:3}, we propose a new mathematical model 
that explains more accurately the unemployment environment
lived in Portugal. The model has one equilibrium point only, 
which is asymptotically stable under some reasonable conditions
(see Theorem~\ref{thm:01}). Intervention policies are investigated 
in Section~\ref{sec:4} from an optimal control perspective. 
We end with Section~\ref{sec:5} of discussion and conclusions.
In appendices, we provide the developed software code for 
the numerical simulations.


\section{Data analysis and the Munoli and Gani 2016 model}
\label{sec:2}

As part of our goal, we collected data from IEFP 
(Instituto do Emprego e Forma\c{c}\~ao Profissional), 
a Portuguese government institution focused on providing education 
and support for the unemployed population \cite{iefp} 
and Bank of Portugal \cite{bport}. We compiled the number of unemployed persons ($U$), 
the unemployment rate (${UR}$) and vacancies available at the IEFP ($D$), 
concerning the time period from January 2004 until June 2016 with monthly frequency. 
A total of 150 observations from each variable were collected, $U_t$, ${UR}_t$, $D_t$, 
$t = 1,\ldots,150$, and three new derived variables were defined:
\begin{itemize}
\item employed people (total number)
\begin{eqnarray}
\label{eq:emp}
E_t = \frac{U_t(1-UR_t)}{UR_t};
\end{eqnarray}
	
\item unemployment change rate
\begin{eqnarray}
\label{eq:unemp}
RCU_t = \frac{U_t-U_{t-1}}{U_{t-1}};
\end{eqnarray}
	
\item employment change rate
\begin{eqnarray}
\label{eq:emp:rate}
RCE_t = \frac{E_t-E_{t-1}}{E_{t-1}}.
\end{eqnarray}
\end{itemize}
Formulas \eqref{eq:unemp} and \eqref{eq:emp:rate} indicate 
the rate of change of unemployed and employed people, respectively. 
We could not obtain the total number of employees in Portugal 
during our time frame. Therefore, we proceeded with an indirect 
calculation \eqref{eq:emp} using the variables that we actually could gather: 
the unemployment rate $UR$ and the total number 
of unemployed people $U$. 

Munoli and Gani (2016) \cite{mg} present a model that tries 
to emulate an unemployment environment. Their model 
consists of three ordinary differential equations, 
considering the unemployed ($U$), 
the employed ($E$), and available vacancies ($V$):
\begin{eqnarray}
\label{eq:M:G}
\begin{cases}
\displaystyle \frac{dU(t)}{dt}
=\Lambda-\kappa U(t) V(t)- \alpha_1 U(t) + \gamma E(t),\\[0.3cm]
\displaystyle \frac{dE(t)}{dt}
=\kappa U(t) V(t)- \alpha_2 E(t) - \gamma E(t),\\[0.3cm]
\displaystyle \frac{dV(t)}{dt}
=\alpha_2 E(t) + \gamma E(t) - \delta V(t) + \phi U(t).
\end{cases}
\end{eqnarray}
The parameters and variables of model \eqref{eq:M:G} are described on Table~\ref{tab:1}. 
\begin{table}[ht]
\caption{\label{tab:1} \small Variables and parameters considered 
by Munoli and Gani (2016) \cite{mg}.}
\centering
\begin{tabular}{c|l}
Variable & Meaning \\\hline
$U(t)$ & Number of unemployed individuals at time $t$ \\
$E(t)$ & Number of employed individuals at time $t$ \\
$V(t)$ & Number of vacancies at time $t$ \\
$\Lambda$ & Number of unemployed individuals that is increasing continuously \\
$\kappa$ & Rate at which the unemployed individuals are becoming employed \\
$\alpha_1$ & Rate of migration and death of unemployed individuals \\
$\alpha_2$ & Rate of retirement as well as death of employed individuals \\
$\gamma$ & Rate of individuals who are fired from their jobs \\
$\phi$ & Represents the rate of creating new vacancies \\
$\delta$ & Denotes the diminution rate of vacancies due to lack of funds
\end{tabular}
\end{table} 
At Munoli and Gani (2016) \cite{mg}, the initial conditions were given
by $U(0) = 10000$, $E (0) = 1000$ and $V (0) = 100$. 
A time line of 150 units ($t = 150$) was considered \cite{mg}. 
Values of the other parameters were fulfilled according with Table~\ref{tab:n2}.  
\begin{table}[ht]
\caption{\label{tab:n2}\small Variables and parameters considered 
by Munoli and Gani (2016) \cite{mg}.}
\centering
\begin{tabular}{c|c|c}
Parameters & Base Value & Reference \\\hline
$\Lambda$ & 5000 & Misra and Singh (2013) \cite{ms2} \\
$\kappa$ & 0.000009 & Misra and Singh (2013) \cite{ms2} \\
$\alpha_1$ & 0.04 & Misra and Singh (2013) \cite{ms2} \\
$\alpha_2$ & 0.05 & Misra and Singh (2013) \cite{ms2} \\
$\gamma$ & 0.001 & Assumed \cite{mg} \\
$\phi$ & 0.007 & Assumed \cite{mg}\\
$\delta$ & 0.05 & Assumed \cite{mg}
\end{tabular}
\end{table} 
We replaced the initial conditions suggested by the authors of \cite{mg}
by the ones given by the real data from Portugal. Precisely, 
we fixed the initial values to the ones of Portugal at January 2004, 
when the number of unemployed people was 464450 ($U$), 
employed was 6450694 ($E$, according with \eqref{eq:emp}) 
and the available vacancies at the time were $4848$ ($V$).  
\begin{figure}[!htb]
\centering
\subfloat[\footnotesize{Employed population}]{\label{Emprego}
\includegraphics[width=0.45\textwidth]{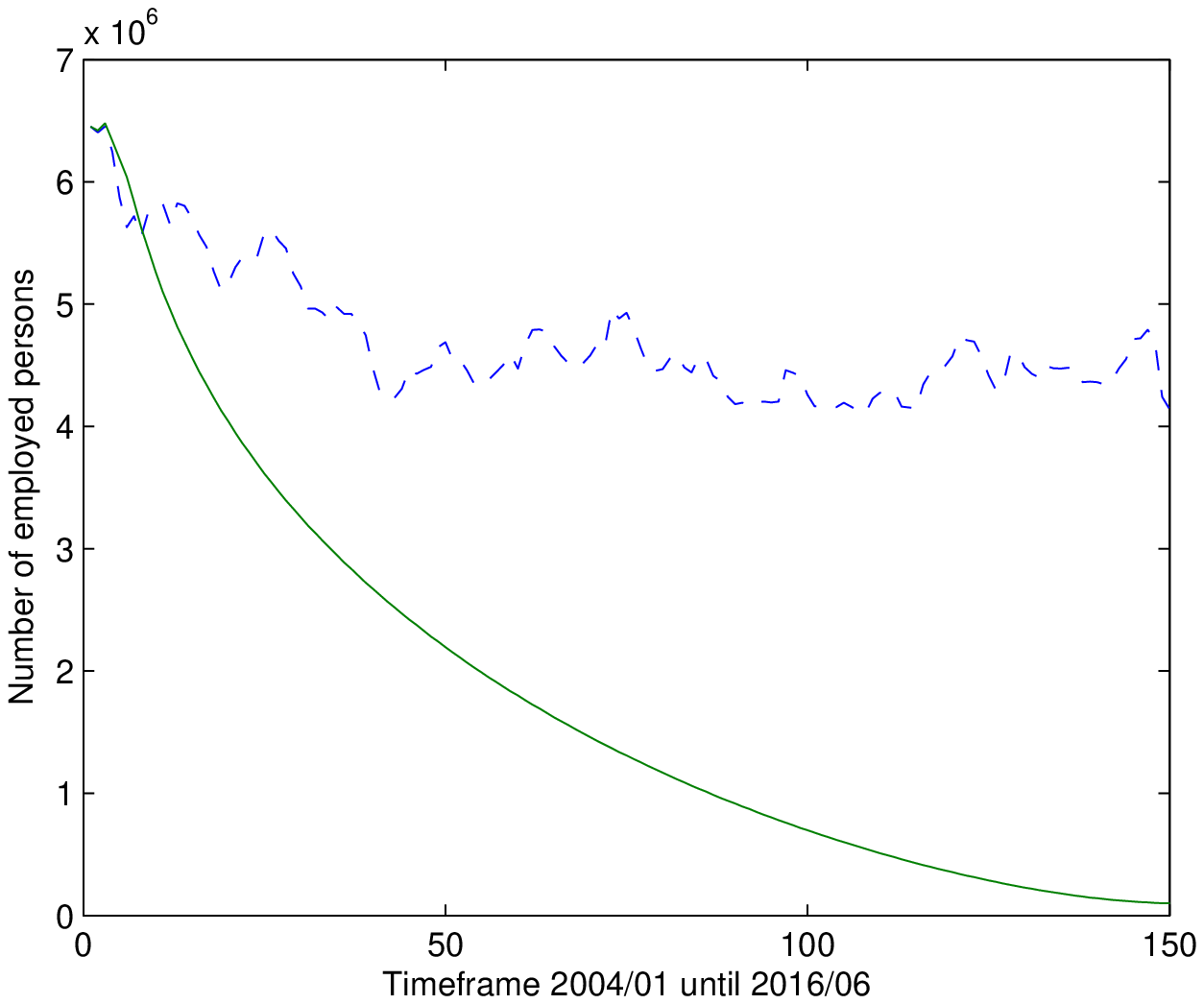}}
\subfloat[\footnotesize{Unemployed population}]{\label{Desemprego}
\includegraphics[width=0.45\textwidth]{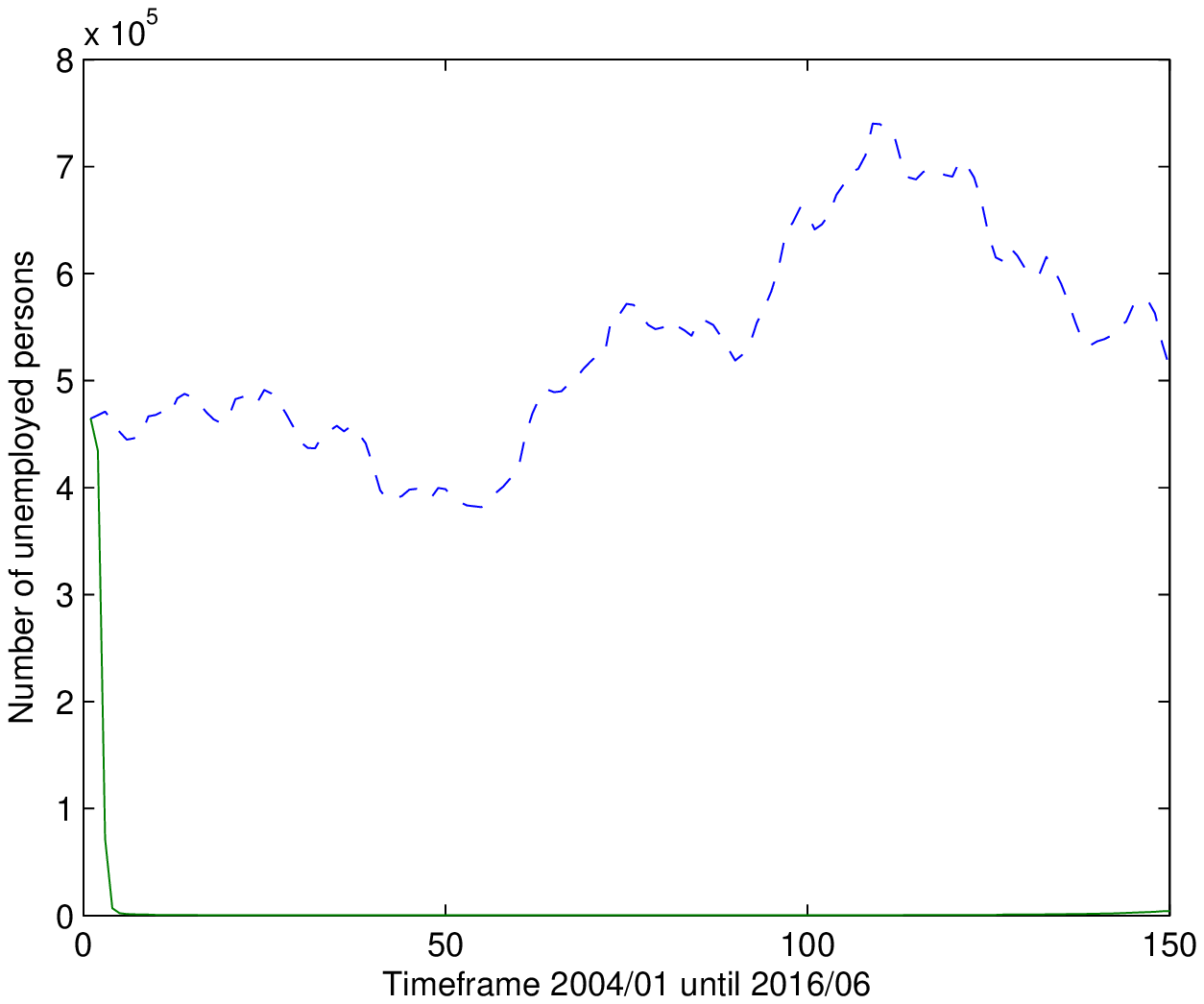}}\\
\subfloat[\footnotesize{Total vacancies}]{\label{Vagas}
\includegraphics[width=0.45\textwidth]{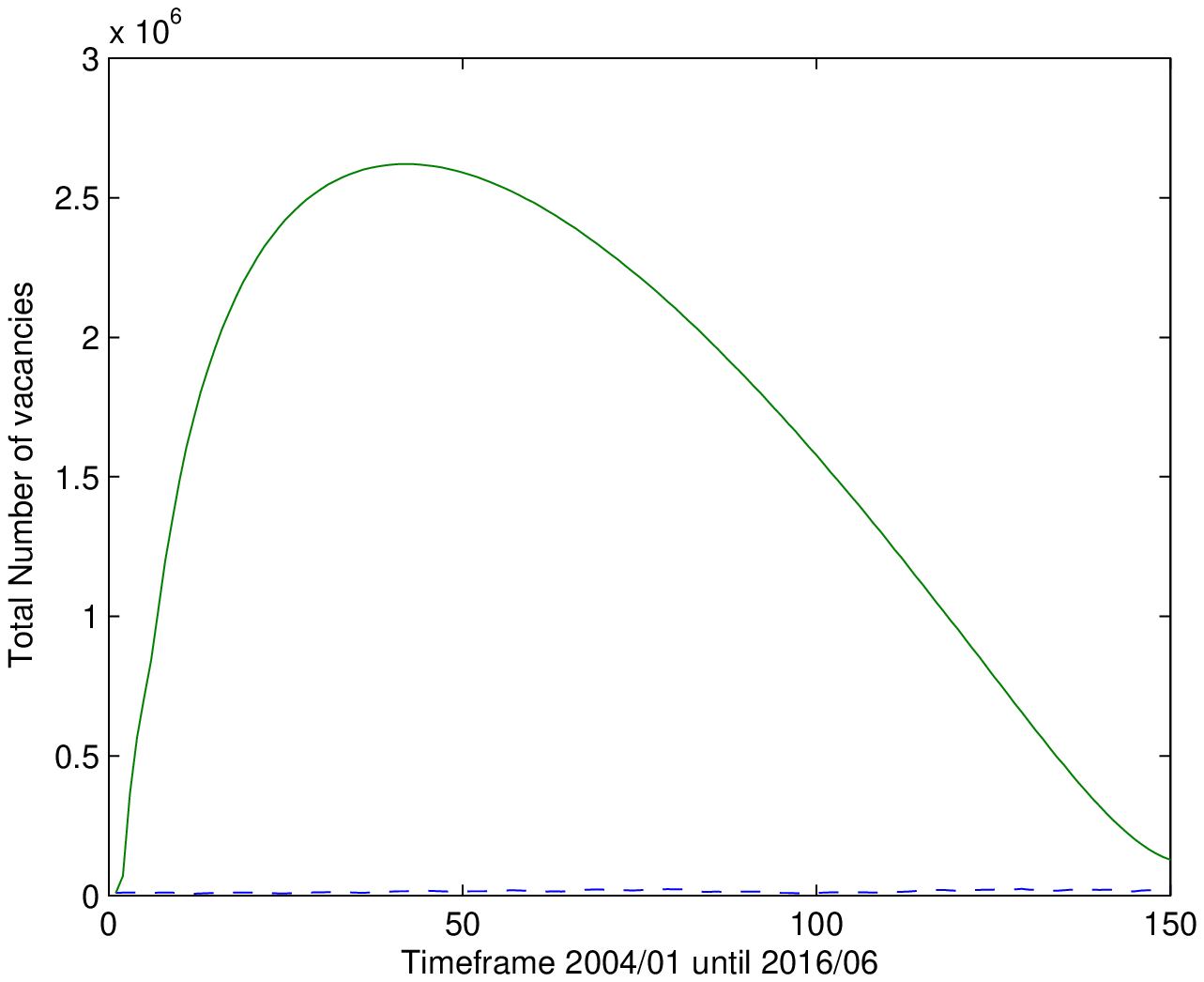}}
\caption{\small Real data (dashed--blue line) versus prediction from Munoli 
and Gani model of 2016 \cite{mg} (continuous--green line).}
\label{fig:01}
\end{figure}
We can easily state that the suggested model from Munoli and Gani (2016) \cite{mg} 
does not replicate properly the real data from Portugal. 
Regarding the employment and unemployment (see Figures~\ref{Emprego} 
and \ref{Desemprego}, respectively) the model of \cite{mg} kind of implode, 
since the unemployed/employment values dramatically drop until the end of the time period. 
Munoli and Gani (2016)~\cite{mg} differential system \eqref{eq:M:G}
also suggests an exceptional increase on the supply of job vacancies 
with a smoother decrease until the end of the time period, a statement that 
is not supported by the Portuguese real data. Actually, the real number of vacancies 
presents a smooth fluctuation and a shy tendency to increase over time
(see Figure~\ref{Vagas}).


\section{New unemployment model and simulations}
\label{sec:3}

Given the weak results of Section~\ref{sec:2}, our main goal is to create 
a new mathematical model that explains more accurately the unemployment environment. 
With this in mind, we proceed with a few changes in Munoli and Gani (2016) \cite{mg} model 
to achieve the desired result. First of all, we consider that the number of vacancies 
should be an exogenous variable and not given by any specific differential equation, 
as stated on previous works. The inherent fluctuation and apparent pattern lead us 
to fit our data in order to achieve a trustworthy representation of this variable. 
We fit our data with a 3rd degree Fourier function (see Appendix~\ref{sec:appendix}), 
obtaining a reasonable goodness of fit (precisely, a R-square of 0.8046, 
as given in Appendix~\ref{sec:appendix}). Figure~\ref{VagasFit} 
shows the fitted function and the actual vacancies data.
\begin{figure}[!htb]
\centering
\includegraphics[scale=0.5]{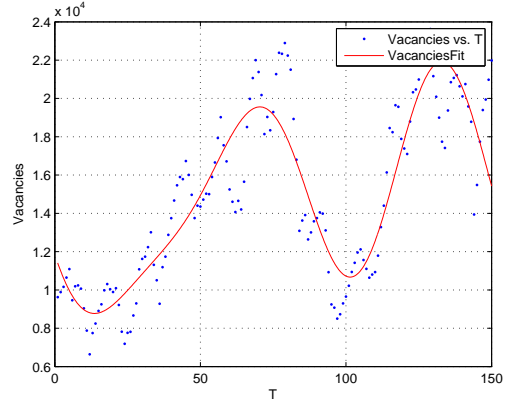}
\caption{\small Total vacancies. Real data (dots) and 
fitting by a 3rd degree Fourier function (continuous line).}
\label{VagasFit}
\end{figure}

Having in mind the results of S\^{\i}rghi et al. of 2014 \cite{sa}, 
we decided to include one variable that compiles the employment created 
due to the increase of the unemployment. Bigger unemployment rates 
signals employers to hire at a lower wage and, as a consequence, 
they are also able to hire more workers. We computed the $p$-values for 
Pearson's correlation using a Student's $t$-distribution for a 
transformation of the correlation using the \texttt{corr(X,Y)} command of 
\textsf{MatLab}. This function is exact when $X$ and $Y$ are normal.  
To retrieve such correlation value, we used the two 
transformed variables $RCU_t$ and $RCE_t$ (formulas \eqref{eq:unemp}
and \eqref{eq:emp:rate}), obtaining the value of 0.7161.    
We also found that the constant rate, at which the number of unemployed persons 
is increasing continuously, as assumed by Munoli and Gani (2016) \cite{mg} 
and Misra and Singh (2013) \cite{ms2}, is quite small regarding the Portuguese population. 
For this reason, we have increased the value to 90000. We also note that the other 
crucial point explaining the under-achievement of previous models  with respect
to Portuguese data (implosion and general shrinkness of the population, see Section~\ref{sec:2}), 
is the absence of a value of a constant rate at which the number of employed persons 
is increasing continuously in order to recoup the loss of people within the system.
Considering the stated above, we propose here the following model for unemployment, 
described by a system of two ordinary differential equations:
\begin{eqnarray}
\label{eq:our:model}
\begin{cases}
\displaystyle \frac{dU(t)}{dt}
=\Lambda-\kappa U(t) V(t)- \alpha_1 U(t) + \gamma E(t),\\[0.3cm]
\displaystyle \frac{dE(t)}{dt}
=\omega+\kappa U(t) V(t)- \alpha_2 E(t) - \gamma E(t)-\delta E(t) +\rho U(t).
\end{cases}
\end{eqnarray}
The meaning of the variables and parameters of our model \eqref{eq:our:model}
is given in Table~\ref{tab:our:model}. The values used in our simulations 
are given in Table~\ref{tab:our:values}.
\begin{table}[ht]
\caption{\label{tab:our:model} \small Variables and parameters 
of our mathematical model \eqref{eq:our:model}.}
\centering
\begin{tabular}{l|l}
Variable & Meaning \\\hline
$U(t)$ & Number of unemployed individuals at time $t$ \\
$E(t)$ & Number of employed individuals at time $t$ \\
$V(t)$ & Total vacancies at time $t$ \\
$\Lambda$ & Number of unemployed individuals that is increasing continuously \\
$\kappa$ & Rate at which the unemployed individuals are becoming employed \\
$\alpha_1$ & Rate of migration and death of unemployed individuals \\
$\alpha_2$ & Rate of retirement or death of employed individuals \\
$\gamma$ & Rate of persons who are fired from their jobs\\
$\omega$ & Number of employment created and fulfilled \\
$\delta$ & Denotes the diminution rate of vacancies due to lack of funds \\
$\rho$ & Rate of employment increase due to labour-force wage devaluation
\end{tabular}
\end{table} 
\begin{table}[ht]
\caption{\label{tab:our:values}\small Parameter values considered 
in the simulations of our model \eqref{eq:our:model}.}
\centering
\begin{tabular}{l|l|l}
Parameters & Base Value & Reference \\\hline
$\Lambda$ & 90000 & Assumed \\
$\kappa$ & 0.000009 & Misra and Singh (2013) \cite{ms2} \\
$\alpha_1$ & 0.04 & Misra and Singh (2013) \cite{ms2} \\
$\alpha_2$ & 0.05 & Misra and Singh (2013) \cite{ms2} \\
$\gamma$ & 0.001 & Munoli and Gani (2016) \cite{ms2} \\
$\omega$ & 90000 & Assumed \\
$\delta$ & 0.05 & Munoli and Gani (2016) \cite{ms2} \\
$\rho$ & 0.7161 & Assumed (according to variable correlation) \\		
\end{tabular}
\end{table} 

The feasible region of model \eqref{eq:our:model} is given by
\begin{equation}
\label{eq:Omega}
\Omega = \left\{(U, E) : 0 \leq U + E \leq \frac{\Lambda + \omega}{\alpha_m}\right\},
\quad \alpha_m = \min\left\{\alpha_1 - \rho, \alpha_2 + \delta\right\}.
\end{equation}
Next result gives the positivity invariance of the feasible region \eqref{eq:Omega}.

\begin{theorem}
The set $\Omega$ defined by \eqref{eq:Omega} is a region 
of attraction for the model system \eqref{eq:our:model}. 
\end{theorem}

\begin{proof}
From the model equations \eqref{eq:our:model}, we get
$$
\frac{dU(t)}{dt} + \frac{dE(t)}{dt} 
= \Lambda + \omega - \alpha_1 U(t) +\rho U(t) 
- \alpha_2 E(t) -\delta E(t),
$$
which gives
$$
\frac{d}{dt} \left[ U(t) + E(t) \right]
\leq \Lambda + \omega - \alpha_m \left[ U(t) + E(t) \right],
$$
where $\alpha_m = \min\left\{\alpha_1 - \rho, \alpha_2 + \delta\right\}$.
Taking the limit supremum, one has
$$
\lim \sup_{t \rightarrow +\infty} \left[ U(t) + E(t) \right]
\leq  \frac{\Lambda + \omega}{\alpha_m}.
$$
This proves our theorem.
\end{proof}

Given model \eqref{eq:our:model} and the parameter 
values of Table~\ref{tab:our:values}, we carried out the simulation 
in \textsf{Matlab} (see Appendix~\ref{sec:appendix}). Since the required $T$ 
to smooth our differential equation was 81 observations, we needed to compress 
our observed real values (150 observations) in order to achieve a 
graphical comparison (see Figure~\ref{fig:5}).
\begin{figure}[!htb]
\centering
\includegraphics[scale=0.5]{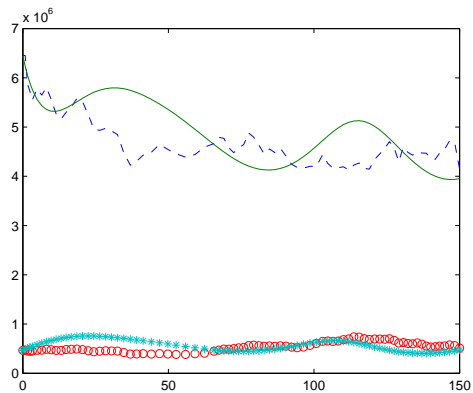}
\caption{\small Unemployed: real data (circle--red), 
model simulation (asterisk--light-blue).
Employed: real data (dashed--blue), model simulation (continuous--green).}
\label{fig:5}
\end{figure} 
As we can see, our model fits the observed data much better 
than the results obtained from previous models available 
in the literature. Even though there are a few opposite fluctuations, 
our model achieves a much more steady environment.


\subsection{Equilibrium analysis}
\label{sec:3.1}

To compute the equilibrium points of the proposed model \eqref{eq:our:model},
one needs to solve the following system:
\begin{eqnarray*}
\begin{cases}
\Lambda-\kappa U V- \alpha_1 U + \gamma E= 0,\\[0.3cm]
\omega+\kappa U V- \alpha_2 E - \gamma E-\delta E +\rho U =0.
\end{cases}
\end{eqnarray*}
Direct calculations show that there exists 
one equilibrium point $E_b=(U^{*},E^{*})$ only, given by
\begin{equation}
\label{eq:Equilibriumpoint}
\begin{gathered}
E^{*}=\frac{\alpha_1\omega+\Lambda\rho+\kappa(\omega+\Lambda)V}{(\alpha_1-\rho)\gamma
+(\alpha_2+\delta)\kappa V + \alpha_1(\alpha_2+\delta)},\\[0.3cm]
U^{*}=\frac{\Lambda(\delta+\alpha_2)+(\omega+\Lambda)\gamma}{(\alpha_1-\rho)\gamma
+(\alpha_2+\delta)\kappa V + \alpha_1(\alpha_2+\delta)}.
\end{gathered}
\end{equation}

\begin{remark}
All the parameters that appear in \eqref{eq:Equilibriumpoint} are strictly positive. 
Therefore, the numerators of $U^{*}$ and $E^{*}$ are always positive. 
The only possibility for $U^{*}$ and $E^{*}$ to be negative 
would be to have 
\begin{eqnarray*}
\alpha_1 < \rho \quad \text{and} \quad
(\alpha_2+\delta)\kappa V + \alpha_1(\alpha_2+\delta) < |(\alpha_1-\rho)\gamma|,
\end{eqnarray*}
which is not a reasonable scenario since the variable $V$, representing the available 
vacancies, is way bigger than all other parameters appearing in the denominators
of $U^{*}$ and $E^{*}$, reasonably valued in the interval $[0, 1]$.
\end{remark}


\subsection{Stability analysis}

We now study the local stability of the equilibrium $E_b$ found 
in Section~\ref{sec:3.1}. To achieve this, we compute the so called 
variational matrix $M$ of our designated model \eqref{eq:our:model}:
\begin{eqnarray}
\label{eq:mat:M}
M=\begin{bmatrix}
-\kappa V - \alpha_1 & \gamma \\
\kappa V + \rho & -\alpha_2-\gamma-\delta
\end{bmatrix}.
\end{eqnarray}
The characteristic equation associated with our $2\times2$ matrix \eqref{eq:mat:M} is
\begin{eqnarray}
\label{eq:cp}
\lambda^2 + a_1 \lambda + a_2 = 0
\end{eqnarray}
with
\begin{equation}
\label{eq:coef}
\begin{split}
a_1 &= V k + \alpha_1 + \alpha_2+\delta+\gamma,\\
a_2 &= V\alpha_2 k+V\delta k+\alpha_1\alpha_2+\alpha_1\delta+\alpha_1\gamma-\gamma\rho.
\end{split}
\end{equation}
The Routh--Hurwitz criterion for second degree polynomials
asserts that \eqref{eq:cp} has all the roots in the left half plane (and the system is stable)
if and only if both coefficients \eqref{eq:coef} are positive \cite{Dorf:Bishop}.
We just proved the following result.

\begin{theorem}
\label{thm:01}
The equilibrium $E_b=(U^{*},E^{*})$ given by \eqref{eq:Equilibriumpoint}
is locally asymptotically stable if and only if
$$
V\alpha_2 k+V\delta k+\alpha_1\alpha_2+\alpha_1\delta+\alpha_1\gamma > \gamma\rho.
$$
\end{theorem}

\begin{remark}
For the values used to describe the Portuguese reality of unemployment, 
one has $a_2 = 0.00000090\,V+ 0.0033239$, which is strictly positive 
because $V$ is always non-negative. It follows that
the unique equilibrium point of our system 
is locally asymptotically stable.
\end{remark}


\section{Optimal control}
\label{sec:4}

Portugal is a country with serious unemployment issues during the last decade
and the government of Portugal was forced to apply intervention policies 
in this particular area. Regarding this subject, the adoption of internships 
and hiring support measures (policies where the government contributes 
with a share of the worker's salary during a pre-established period, normally one year) 
have been in force since 1997, with variable magnitude until nowadays. 
Facing more severe unemployment problems at the 2007/2008 crisis, those measures 
became quite popular as a fundamental axle regarding the battle against unemployment 
and integration in the labour market of the recent graduates. With reference 
to the bibliography on this subject, there are a few empirical works that 
try to answer or explain the following difficult question: ``\emph{Does the supply 
of internships fight the long-term unemployment?}'' 
The work of Silva et al. (2016) \cite{silva} focus on
the impact of the internships in the unemployment of graduates compared 
to other age-similar groups. Another study, Barnwell (2016) \cite{Barn}, 
addresses the effectiveness of the internship component in the increasing 
employability of graduates. However, we are not aware of any empirical work 
that responds concretely to the aforementioned question. Using our representative model 
of the labor market reality, we now introduce two controls into \eqref{eq:our:model}:
\begin{eqnarray}
\label{eq:our:optmodel}
\begin{cases}
\displaystyle \frac{dU(t)}{dt}
=\Lambda-\kappa U(t) V(t)(1+u_2(t))- \alpha_1 U(t) + \gamma E(t) -u_1(t),\\[0.3cm]
\displaystyle \frac{dE(t)}{dt}
=\omega+\kappa U(t) V(t) (1+u_2(t))- \alpha_2 E(t) - \gamma E(t)-\delta E(t) +\rho U(t)+u_1(t).
\end{cases}
\end{eqnarray}
The first control function, $u_1$, refers to the unitary supply of internships or support measures; 
while the second control function, $u_2$, represents other alternative indirect measures such as 
lowering corporate tax rates. The offer of an internship, represented by the control measure $u_1$, 
has a direct or immediate impact based on the simple premise that an unemployed worker 
shifts to the employed group due to this incentive. This variable is scaled between 
$-40000$ and $40000$ because it is possible to add or withdraw internships 
already operating in the market. The cost of each internship is registered 
by the monetary value of the support plus all the administrative costs 
inherent to the planning and execution of the internship. 
The magnitude of indirect benefits, denoted by the control variable $u_2$, 
interacts directly with the exogenous function since those measures affect 
the natural creation of employment. We settled its value between $0$ and $1$, 
being its cost interpreted in the monetary unit (m.u.) of internships. 
The optimal control problem is as follows:
\begin{eqnarray}
\label{eq:our:funcional}
J[U(\cdot),u_1(\cdot),u_2(\cdot)]
=\int_0^{150} [A (U(t)-U(0)) + B u_1(t) + C u_2(t)] dt \longrightarrow \min
\end{eqnarray}
subject to the constraints on the control values
\begin{eqnarray}
\label{OCP:contCosnt}
-40000 \leq u_1(t) \leq 40000, \quad
0 \leq u_2(t) \leq 1, 
\quad t \in [0, 150],
\end{eqnarray}
the initial conditions
\begin{eqnarray}
\label{OCP:IC}
U(0)=464450, \quad E(0)=6450694,
\end{eqnarray}
the terminal conditions
\begin{eqnarray}
\label{OCP:TC}
5000000 \leq U(150)+E(150) \leq 8000000
\end{eqnarray}
and the state constraint
\begin{eqnarray}
\label{OCP:SC}
\frac{U(t)}{E(t) + U(t)} \leq 0.12,
\quad t \in [0, 150].
\end{eqnarray}
Since we know that the supply/withdraw of a unitary control $u_1$ represents 
a new employed/unemployed in the system, and a government financial cost/gain 
inherent to this measure, we set $B$ as the reference value equal to $1$. 
In order to settle the value of $A$, we establish an ideal ratio of 20 to 1, 
stating that an unemployed person has the similar cost as 19 new employees, 
using 5\% as a target of the utopian level of unemployment. The value of 
$C$ is set at 40000, representative of 40000 m.u. expressed in internship currency 
so that when $u_2$ is equal to 1 (maximum value) we are stating that we are 
investing the value/cost of 40000 internships in indirect measures. The constraint
\eqref{OCP:SC} keeps the unemployment rate below 12\% while \eqref{OCP:TC} assures
a reasonable level of active labour force between 5 and 8 million people. 
We have chosen the value 12\% since this is the minimum unemployment rate 
in the last 5 years of our real data. The initial conditions \eqref{OCP:IC} 
of employment/unemployment level agree with collected data.

The optimal control problem \eqref{eq:our:funcional}--\eqref{OCP:SC}
is far from trivial and we were not able to find its analytical solution 
through application of the Pontryagin Maximum Principle for optimal
control problems with state constraints. Instead, we solved 
the optimal control problem \eqref{eq:our:funcional}--\eqref{OCP:SC}
numerically using the ACADO Toolkit \cite{Houska2011a} -- see Appendix~\ref{sec:ACADO}.
The results are given in Figures~\ref{fig:7} and \ref{fig:6}.
As we see, looking to the graphics in Figure~\ref{fig:7}, 
\begin{figure}[!htb]
\centering
\includegraphics[scale=0.5]{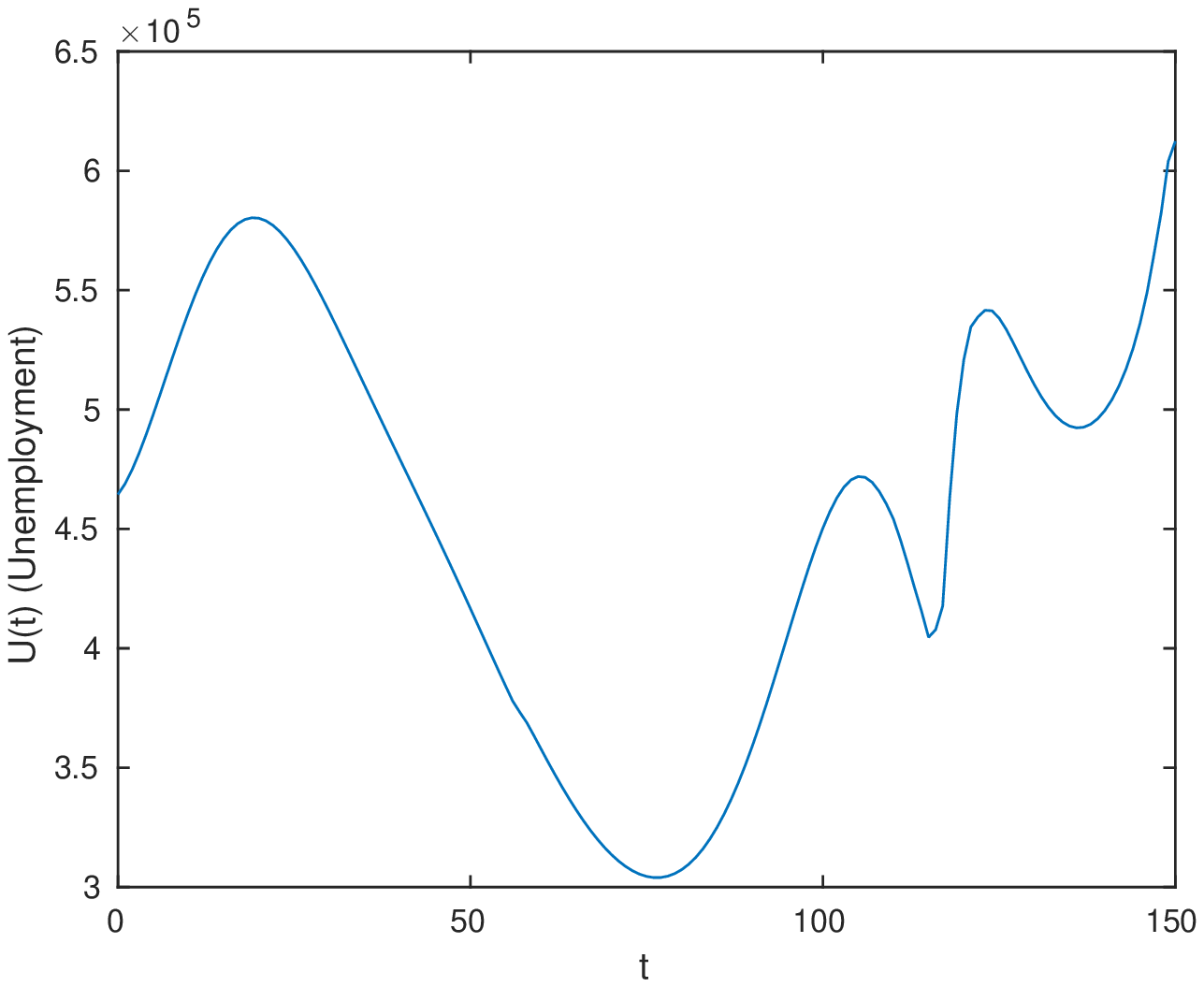} 
\includegraphics[scale=0.5]{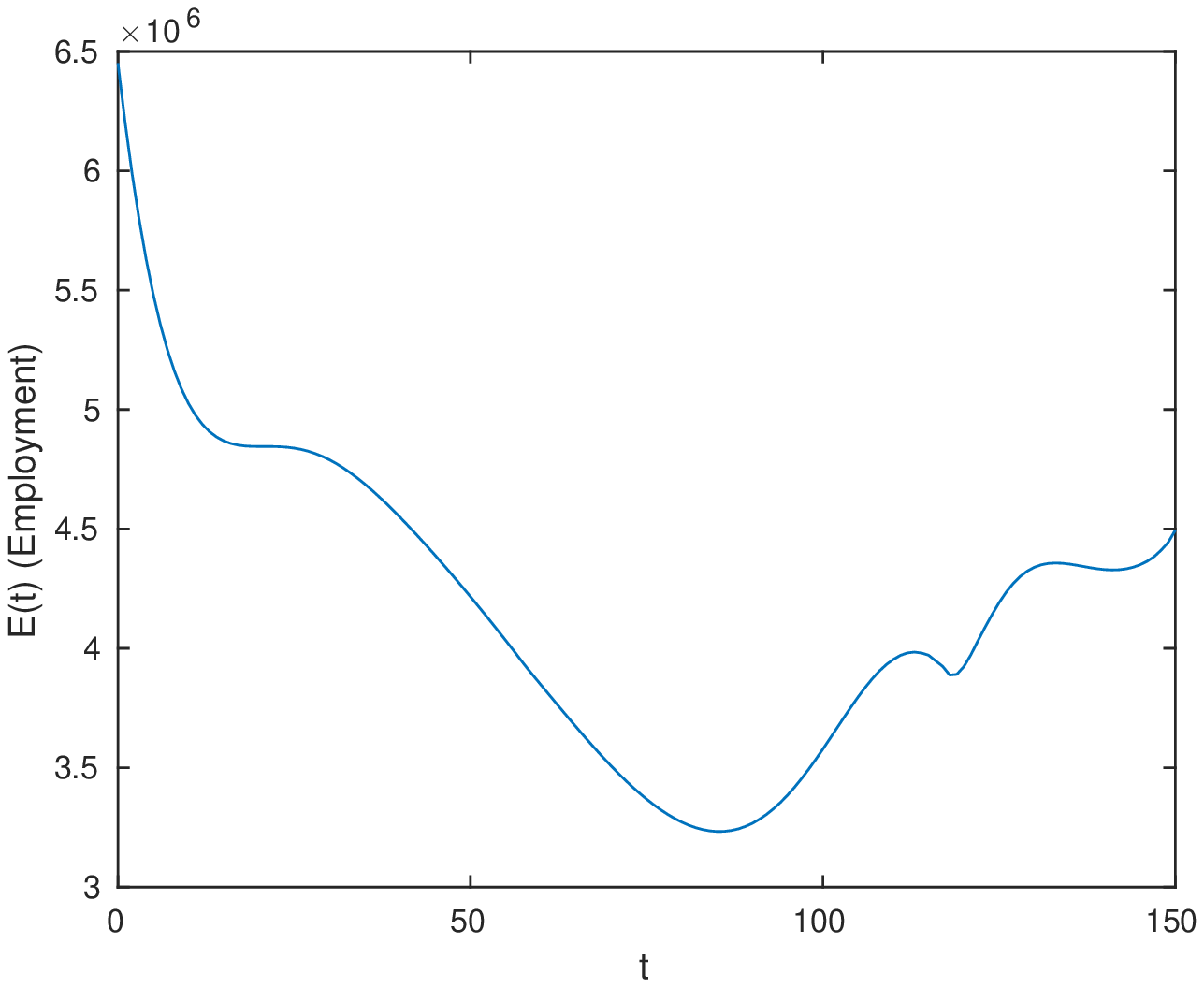}\\
\includegraphics[scale=0.5]{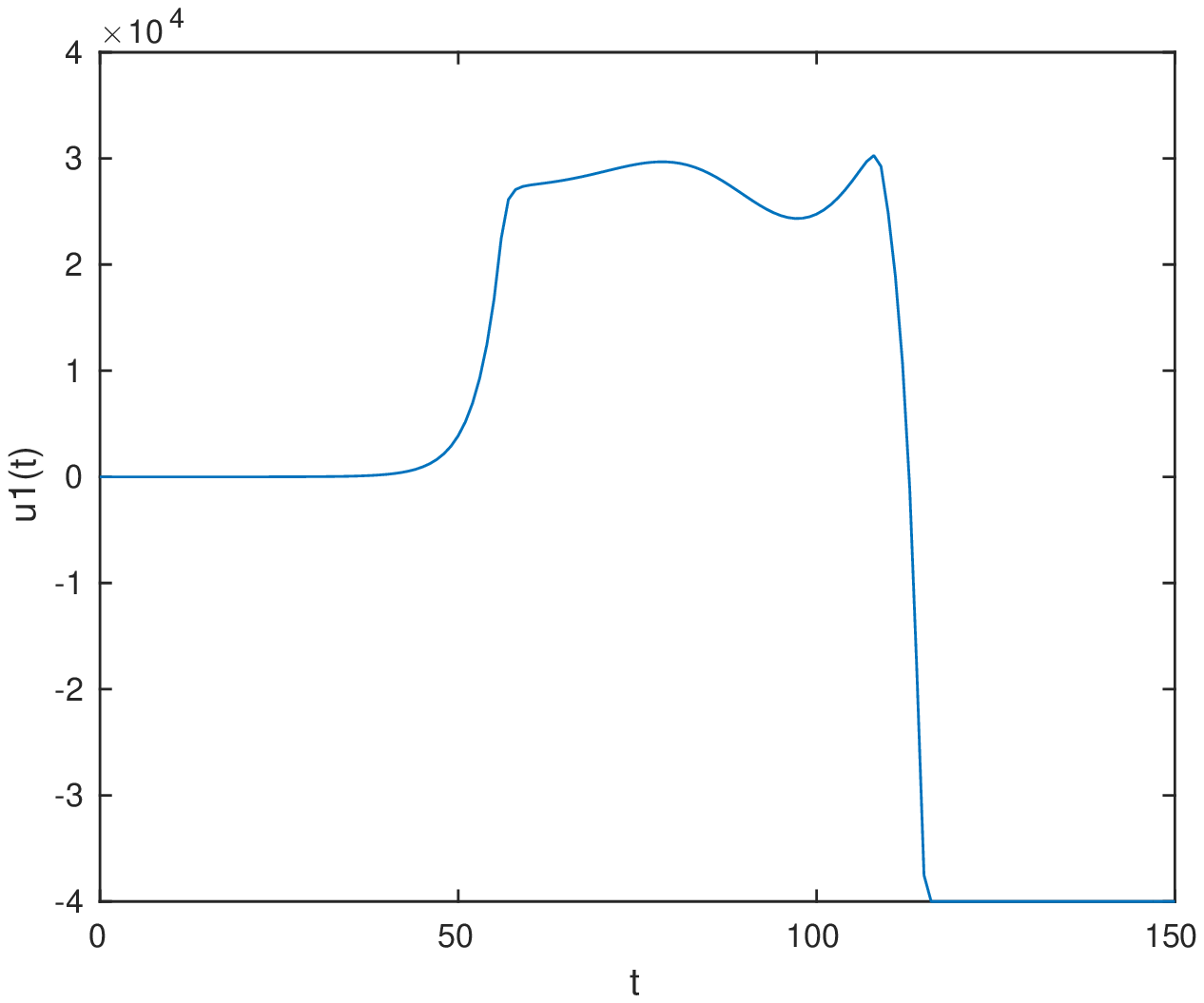}
\includegraphics[scale=0.5]{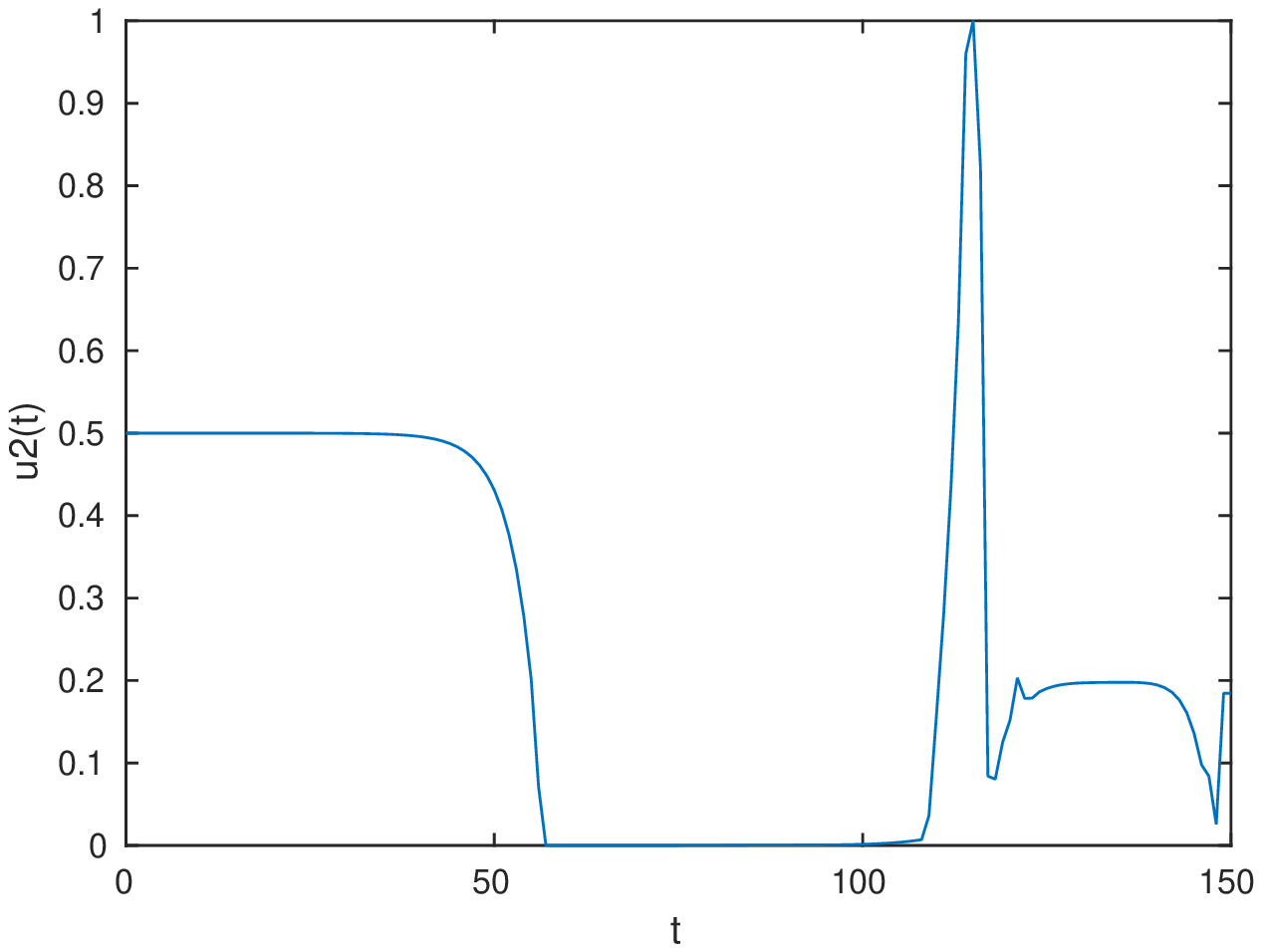}\\
\includegraphics[scale=0.5]{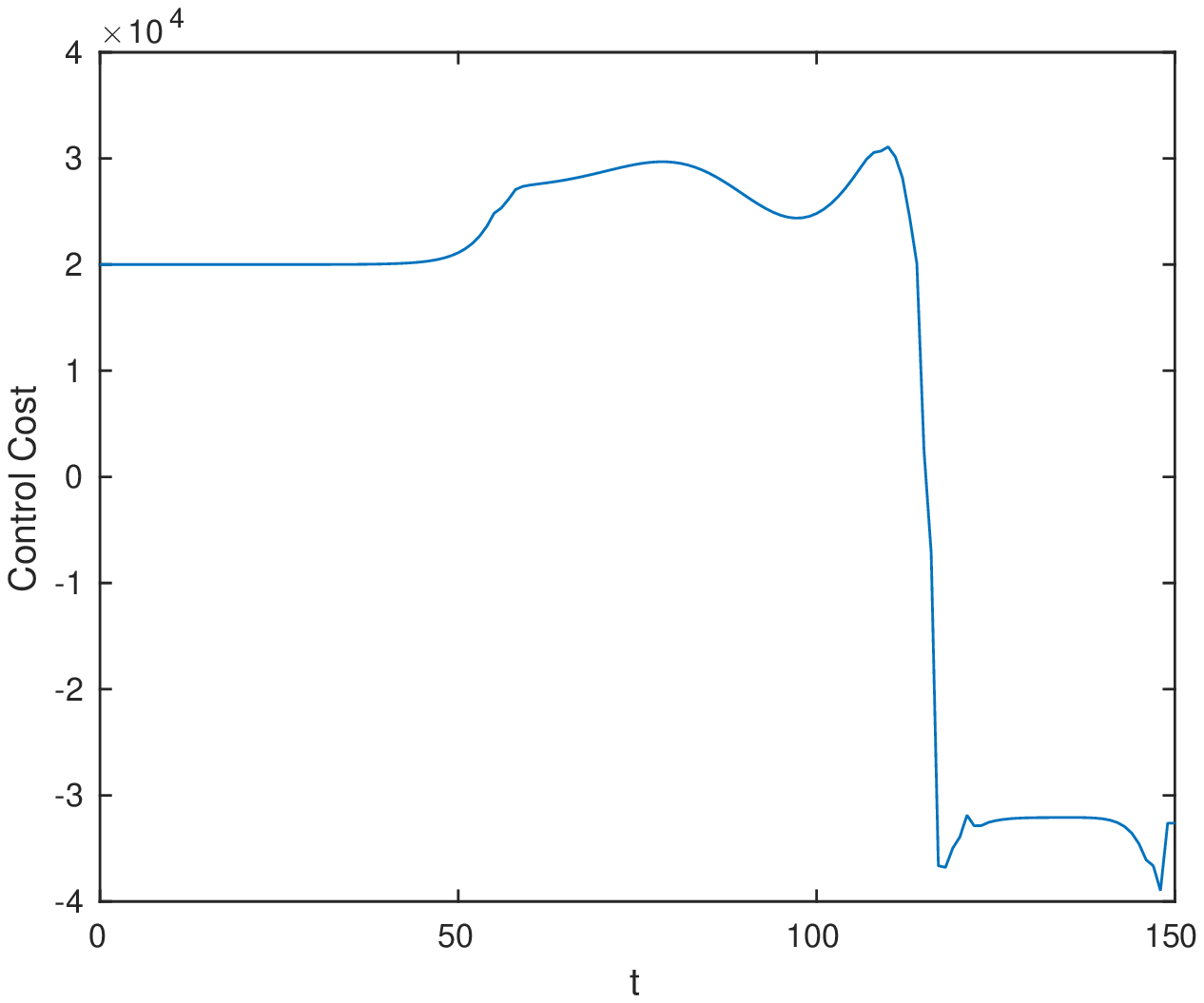}
\includegraphics[scale=0.5]{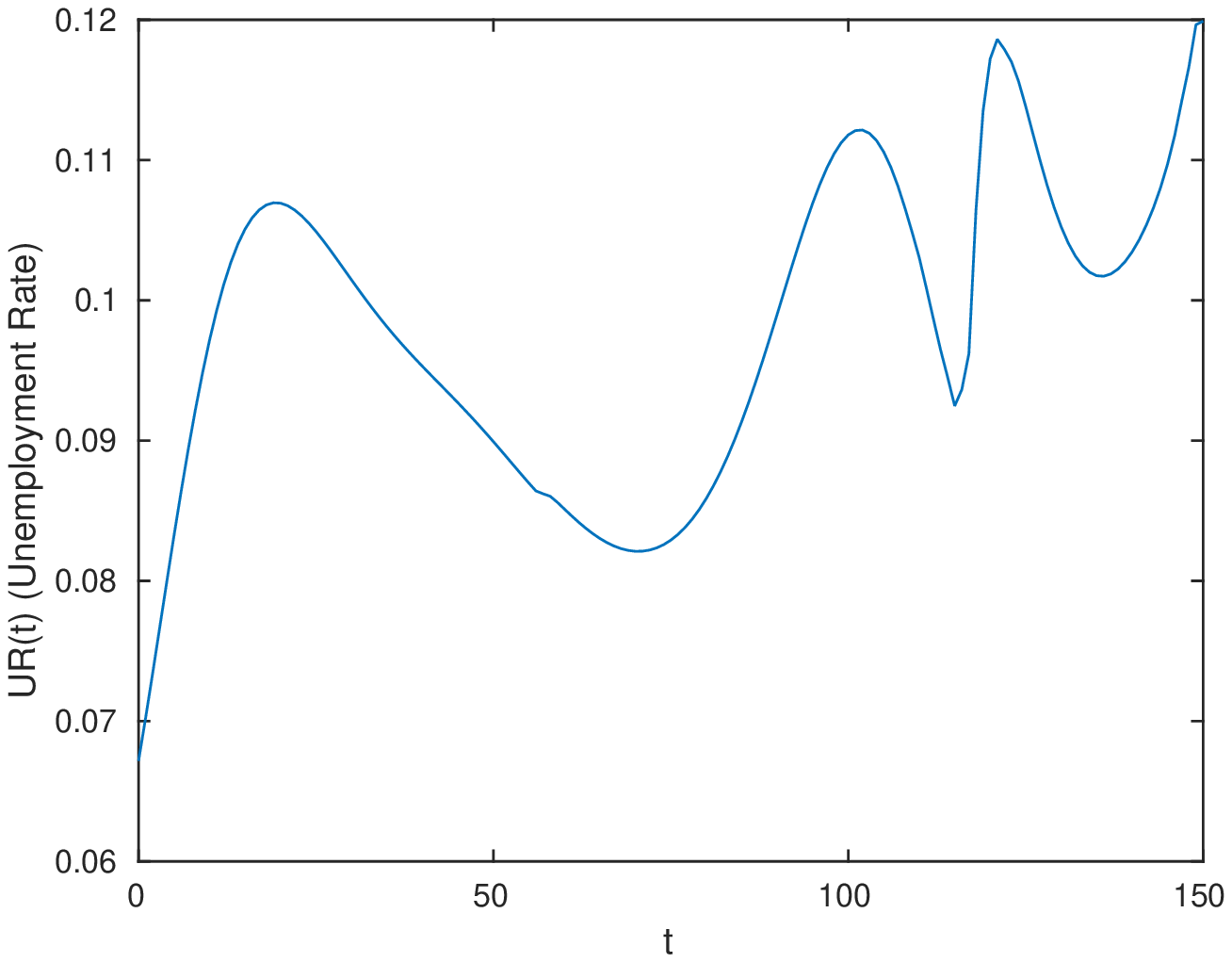}
\caption{\small Numerical solution to the optimal 
control problem \eqref{eq:our:funcional}--\eqref{OCP:SC}
with $A = 20$, $B = 1$ and $C = 40000$.}
\label{fig:7}
\end{figure} 
the model suggests a moderate (about $0.5$) adoption of the $u_2$ control 
from the beginning of the period until $t=50$, point at which the number 
of unemployed people approaches its minimum in our study. At this point 
it is suggested to switch the selected control: the $u_2$ control shrinks 
to zero while a substantial enlarged policy of internships ($u_1$) is suggested 
during the time-frame from $t=50$ until approximately $t=110$. 
The reason for this policy, during the period $t \in [70, 110]$, 
might be related to the employment minimum value. Finally, 
from $t=110$ until the end of the simulation, we assist to a new rise in unemployment levels. 
The optimal control approach points out a brief massive ($u_2=1$) followed by a moderate new
supply of indirect incentives $u_2$ (approximately 0.2) 
and an internship total contraction, that is, $u_1=-40000$. The total cost of applying 
the controls during our time-frame suggests a slight increase in the total investment 
up to $t=110$ culminating in a final saving and financial recovery. We note that
the application of controls slightly tip the expenditure in approximately 
1000 more internships (monthly) during the 150 month study period. 
\begin{figure}[!htb]
\centering
\includegraphics[scale=0.5]{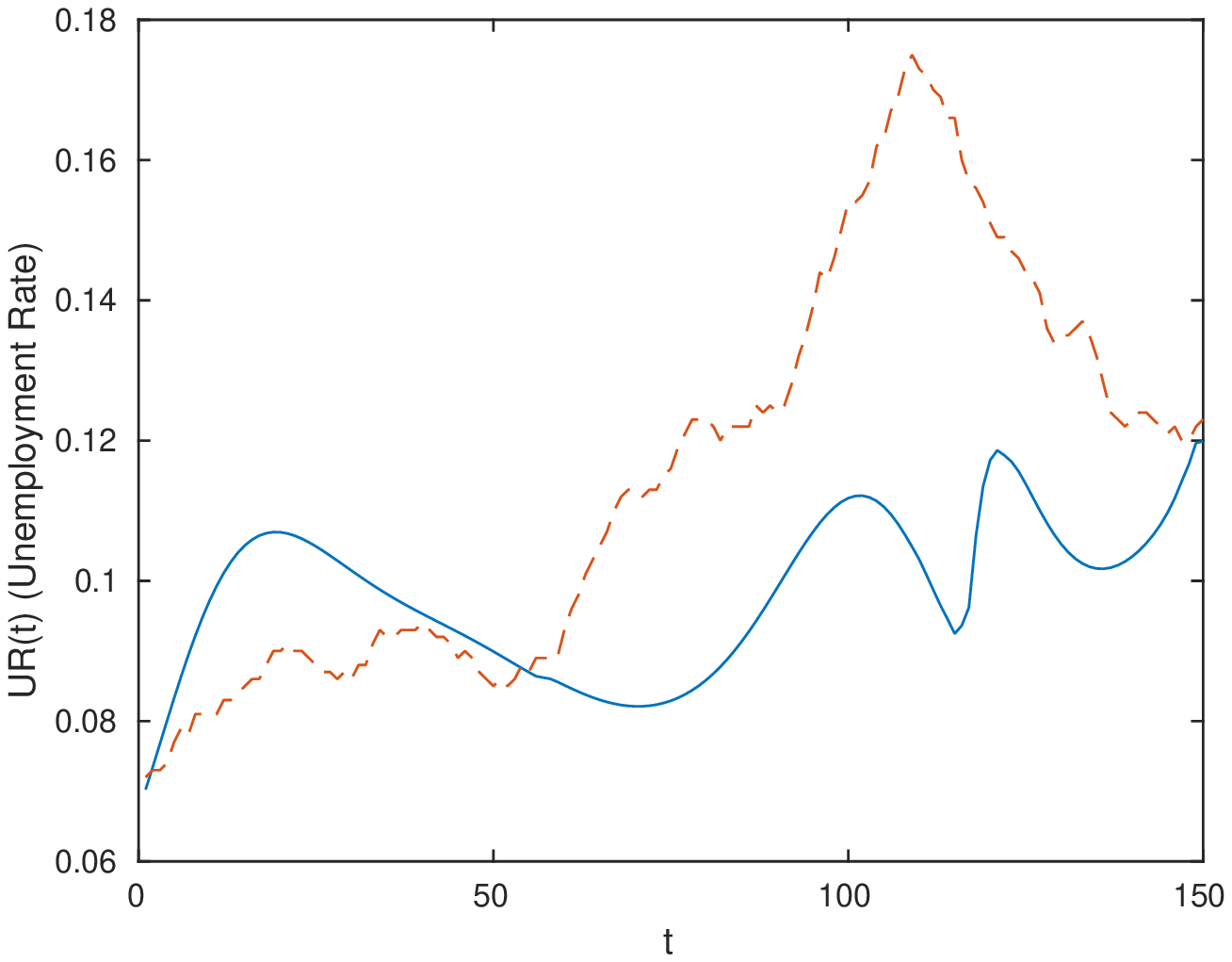}
\caption{\small Unemployment rate: real data (dashed line) 
versus optimal control simulation (continuous-line).}
\label{fig:6}
\end{figure}   
Finally, comparing the actual data with our optimal control simulation, the level of unemployment 
regarding the simulation surpasses slightly the real data from the beginning of the simulation 
until $t=60$. During the remaining period, the optimal control approach avoids the severe 
unemployment period from the actual Portuguese data between $t=100$ and $t=110$, 
which scores values above 17\%. Overall, and considering the mean over the analysed period, 
the simulation with optimal controls scores 9.8\% slightly better than the real data 
value of 11.5\%, as we see in Figure~\ref{fig:6}.


\section{Conclusion}
\label{sec:5}

In this work, we propose a simple mathematical model 
that describes more accurately the real data of unemployment 
from Portugal in the period 2004--2016. 
At first, we considered the proposed model from Munoli and Gani (2016) \cite{mg}, 
fitting their initial conditions with our data. We found out that the model 
suggested by Munoli and Gani in 2016 does not describes at all the Portuguese 
unemployment environment. Because of this, we performed 
a few changes to the Munoli and Gani model \cite{mg}, eliminating one 
of the differential equations, adding an unemployment/wage correlation 
feature, and adjusting the natural rate of unemployed/employed people 
in order to control and balance the inner population of the model. 
Simulations of our model show a much more accurate emulation of the 
Portuguese unemployment environment. Our results also show a slight 
decrease on the overall labour force (unemployed plus employed) and, 
since the Portuguese total population smoothly increased over the last decade, 
that might signals an higher percentage of inactive population, 
which may underpin pressure on more social protection measures. 
Other reading may consist on the premise that the number of unknown unemployed 
people is increasing and standing apart from the government official records.
From the application of optimal control, we can state the following
interesting conclusions: the indirect policies should be the predominant 
method of avoiding unemployment, whereas the supply of internships should 
be the main choice when the total level of employment offered is low. 
A possible explanation why to avoid internships in high unemployment 
periods might be correlated to the severe labour devaluation 
(considered when we developed the considered model) during these seasons, 
since the available labour force is already cheap enough and
the internships supply might be offering jobs with an increased 
government expenditure, due to the whole administrative costs, 
to people that might turned out to be employed anyway.
Moreover, our study also suggests that governmental policies 
should be performed mainly during favourable periods 
in order to avoid high unemployment levels in future crisis.   

As future work, we plan to apply different methods of optimal control
to our model in order to seek even more solid ways and tools to control 
the unemployment issue. For that, we may consider to split the unemployment 
class into two: unemployed that currently receive welfare from government; 
and unemployed which do not receive any financial support. 
These two classes present different government concerns: 
the first one emphasizes financial pressure and the second one social 
and well-being pressure. Another interesting study might be to include 
non-active population, like retired people, and study the optimal control 
regarding the social security financial health.


\section*{Acknowledgements}

The authors were supported by the \emph{Center for Research
and Development in Mathematics and Applications} (CIDMA)
of the University of Aveiro, through Funda\c{c}\~ao 
para a Ci\^encia e a Tecnologia (FCT), 
within project UID/MAT/04106/2013. 
They are very grateful to two referees, for several
comments, suggestions and questions, which helped them
to improve the quality of the paper.



\appendix


\section{MatLab code}
\label{sec:appendix}

We provide here our \textsf{MatLab} code for the simulation 
of the Munoli and Gani (2016) model \cite{mg} 
with Portuguese initial data:
{\small
\begin{verbatim}
f = @(t,x) [5000-0.000009*x(1)*x(3)-0.04*x(1)+0.001*x(2);
0.000009*x(1)*x(3)-0.05*x(2)-0.001*x(2);
0.05*x(2)+0.001*x(2)-0.05*x(3)+0.007*x(1)];
[t,xa] = ode45(f,[0 150],[464450 6450694 9625]);
plot(t,xa(:,1),t,xa(:,2),t,xa(:,3))
xlabel('t'),ylabel('[U,E,V](t)')
\end{verbatim}}
\noindent We obtained a $1737\times 3$ matrix (denoted above by \texttt{xa}).
Thus, we formatted the space in order to make the comparison 
with our data (a $150 \times 3$ matrix):
{\small
\begin{verbatim}
fxa = linspace(1,1737,150)
form = round(fxa)
nxa = xa(form,:)
\end{verbatim}}
\noindent After getting our new $3 \times 150$ matrix (we call it \texttt{nxa}),
we defined the time frame vector \texttt{T}:
{\small
\begin{verbatim}
T = [1:150]
\end{verbatim}}
\noindent The graphical comparison with Munoli and Gani (2016) model \cite{mg}
is then obtained. In what follows, the variables \texttt{DesemGlo}, \texttt{EmpreGlo} 
and \texttt{OfertasGlo} are vectors of dimension 150 (from January 2004 until June 2016), 
denoting, respectively, the total number of unemployed, employed and total vacancies available:
{\small
\begin{verbatim} 
plot(T,DesemGlo,T,nxa(:,1))
xlabel('Timeframe 2004/01 until 2016/06'),ylabel('Number of unemployed persons')
plot(T,EmpreGlo,T,nxa(:,2))
xlabel('Timeframe 2004/01 until 2016/06'),ylabel('Number of employed persons')
plot(T,OfertasGlo,T,nxa(:,3))
xlabel('Timeframe 2004/01 until 2016/06'),ylabel('Total Number of vacancies)
\end{verbatim}}
\noindent Using the MatLab fitting tool, we find a 3rd degree Fourier function 
that fits the total vacancies real data quite well:
{\small
\begin{verbatim}  
General model Fourier3:
f(x) =  a0 + a1*cos(x*w) + b1*sin(x*w) + a2*cos(2*x*w) + b2*sin(2*x*w) 
+ a3*cos(3*x*w) + b3*sin(3*x*w)
Coefficients (with 95% confidence bounds):
a0 =   1.478e+04  (1.444e+04, 1.512e+04)
a1 =       -1262  (-1841, -683.7)
b1 =       -2006  (-2469, -1543)
a2 =       328.2  (-988.4, 1645)
b2 =       -4700  (-5169, -4231)
a3 =       -1992  (-2474, -1510)
b3 =       2.399  (-1202, 1206)
 w =     0.04009  (0.03864, 0.04153)

Goodness of fit:
SSE: 5.995e+08
R-square: 0.8046
Adjusted R-square: 0.795
RMSE: 2055
\end{verbatim}}
{\small
\begin{verbatim}  
function [fitresult, gof] = createFit(T, Vacancies)
%CREATEFIT(T,VACANCIES)
%  Create a fit.
%
%  Data for 'VacanciesFit' fit:
%      X Input : T
%      Y Output: Vacancies
%  Output:
%      fitresult : a fit object representing the fit.
%      gof : structure with goodness-of fit info.
%
% Fit: 'VacanciesFit'.
[xData, yData] = prepareCurveData( T, Vacancies );
% Set up fittype and options.
ft = fittype( 'fourier3' );
opts = fitoptions( ft );
opts.Display = 'Off';
opts.Lower = [-Inf -Inf -Inf -Inf -Inf -Inf -Inf -Inf];
opts.StartPoint = [0 0 0 0 0 0 0 0.0421690289072455];
opts.Upper = [Inf Inf Inf Inf Inf Inf Inf Inf];

% Fit model to data.
[fitresult, gof] = fit( xData, yData, ft, opts );

% Plot fit with data.
figure( 'Name', 'VacanciesFit' );
h = plot( fitresult, xData, yData );
legend( h, 'Vacancies vs. T', 'VacanciesFit', 'Location', 'NorthEast' );
% Label axes
xlabel( 'T' );
ylabel( 'Vacancies' );
grid on

corr(RCU,RCE)
\end{verbatim}}


\section{ACADO code}
\label{sec:ACADO}

For the numerical solution of the optimal control problem 
\eqref{eq:our:funcional}--\eqref{OCP:SC}
described in Section~\ref{sec:4}, we used the ACADO Toolkit,  
which is a free software environment and algorithm collection 
for automatic control and dynamic optimization \cite{Houska2011a}:

{\small
\begin{verbatim}
#include <acado_toolkit.hpp>
#include <acado_gnuplot.hpp>

int main( ){

USING_NAMESPACE_ACADO

DifferentialState    x1,x2,x3 ; // the differential states
Control              u1,u2    ;     
IntermediateState    mu       ;

const double t_start = 0.0;
const double t_end   = 150.0;
const double T       = t_end - t_start;
const double a0  =  1.478e+04;  
const double a1  =  -1262;  
const double b1  =  -2006;  
const double a2  =   328.2;  
const double b2  =  -4700;  
const double a3  =  -1992;  
const double b3  =   2.399;  
const double w   =   0.04009; 

DifferentialEquation f( 0.0, T );     

//  -------------------------------------

OCP ocp(t_start,t_end,150);   

ocp.minimizeLagrangeTerm( 20*x1 - 20*x1(0) + u1 + 40000*u2 );

mu=a0+a1*cos(x3*w)+b1*sin(x3*w)+a2*cos(2*x3*w)
     +b2*sin(2*x3*w)+a3*cos(3*x3*w)+b3*sin(3*x3*w);

f << dot(x1) == 90000-(1+u2)*0.000009*x1*mu-0.04*x1+0.001*x2-u1;
f << dot(x2) == 50000+(1+u2)*0.000009*x1*mu-0.05*x2-0.06*x2-0.001*x2+0.7161*x1+u1;              
f << dot(x3) == 1;

ocp.subjectTo( f );  // minimize T s.t. the model
ocp.subjectTo( AT_START, x1 == 464450  );     
ocp.subjectTo( AT_START, x2 ==  6450694.0 );
ocp.subjectTo( AT_START, x3 ==  0 );
ocp.subjectTo( AT_END, 5000000 <= x1 + x2 <=  8000000 );
ocp.subjectTo( -40000 <= u1 <=  40000   );
ocp.subjectTo( 0 <= u2 <=  1   );
ocp.subjectTo( x1/(x1+x2) <=  0.12   );

//  -------------------------------------

GnuplotWindow window;
window.addSubplot( x1, "Unemployment" );
window.addSubplot( x2, "Employment"   );
window.addSubplot( u1, "u1"           );
window.addSubplot( u2, "u2" );
window.addSubplot( u1 + 40000*u2 , "Control Cost");
window.addSubplot( x1/(x1+x2), "unemployment rate");

OptimizationAlgorithm algorithm(ocp); // the optimization algorithm
algorithm.set( HESSIAN_APPROXIMATION, CONSTANT_HESSIAN );
algorithm.set( KKT_TOLERANCE,1e-2 );
algorithm.set(MAX_NUM_ITERATIONS, 25); 
algorithm << window;
algorithm.solve();                    // solves the problem.

algorithm.getDifferentialStates("states.csv");
algorithm.getControls("ctrl.csv");

return 0;
}
\end{verbatim}}



\begin{thebibliography}{xx}

\bibitem{Barn}
\newblock S.~Barnwell, 
\newblock \emph{Relationship between internships and employment competencies 
of degreed professionals who completed a college internship},
\newblock PhD thesis, Walden University, 2016.
\newblock \url{http://scholarworks.waldenu.edu/dissertations/2917/}

\bibitem{Dorf:Bishop}
\newblock R. C. Dorf and R. H. Bishop,
\newblock \emph{Modern Control Systems}, 
\newblock Prentice Hall, 2001.

\bibitem{hn} 
\newblock L.~Harding and M.~Neam\c{t}u,
\newblock \emph{A dynamic model of unemployment with migration and delayed policy intervention}, 
\newblock Comput. Econ., in press. 
\newblock DOI: 10.1007/s10614-016-9610-3	

\bibitem{Houska2011a}
\newblock B. Houska and H.~J. Ferreau and M. Diehl,
\newblock \emph{ACADO Toolkit -- An Open Source Framework for Automatic Control and Dynamic Optimization},
\newblock Optimal Control Appl. Methods {\bf 32} (2011), no.~3, 298--312.

\bibitem{ms1} 
\newblock A.~K.~Misra and A.~K.~Singh, 
\newblock \emph{A mathematical model for unemployment}, 
\newblock Nonlinear Anal. Real World Appl. {\bf 12} (2011), no.~1, 128--136.

\bibitem{ms2} 
\newblock A.~K.~Misra and A.~K.~Singh, 
\newblock \emph{A delay mathematical model for the control of unemployment}, 
\newblock Differ. Equ. Dyn. Syst. {\bf 21} (2013), no.~3, 291--307.

\bibitem{mg} 
\newblock S.~B.~Munoli and S.~Gani, 
\newblock \emph{Optimal control analysis of a mathematical model for unemployment}, 
\newblock Optimal Control Appl. Methods {\bf 37} (2016), no.~4, 798--806. 

\bibitem{nt} 
\newblock C.~V.~Nikolopoulos and D.~E.~Tzanetis, 
\newblock \emph{A model for housing allocation of a homeless population due to a natural disaster}, 
\newblock Nonlinear Anal. Real World Appl. {\bf 4} (2003), no.~4, 561--579. 

\bibitem{silva}
\newblock P.~Silva, B.~Lopes, M.~Costa, A.~I.~Melo, G.~P.~Dias, E.~Brito and D.~Seabra, 
\newblock \emph{The million-dollar question: can internships boost employment?}, 
\newblock Studies in Higher Education, in press.
\newblock DOI: 10.1080/03075079.2016.1144181

\bibitem{sa} 
\newblock N.~S\^{\i}rghi, M.~Neam\c{t}u and D.S.~Deac, 
\newblock \emph{A dynamic model for unemployment control with distributed delay},
\newblock Math. Meth. Fin. Bus. Admin. (2014), 42--48.

\bibitem{bport}
\newblock WWW, 
\newblock Banco de Portugal,
\newblock \url{https://www.bportugal.pt/}

\bibitem{iefp}
\newblock WWW, 
\newblock Instituto do Emprego e Forma\c{c}\~ao Profissional, 
\newblock \url{https://www.iefp.pt/}

\end{thebibliography}
\end{document}